\documentclass[11pt]{amsart}
\usepackage{amssymb}
\usepackage{amsmath}
\usepackage{amsfonts}
\usepackage{amscd}

\newtheorem{theorem}{Theorem}[section]
\newtheorem{definition}[theorem]{Definition}

\newtheorem{corollary}[theorem]{Corollary}
\theoremstyle{definition}
\theoremstyle{remark}
\numberwithin{equation}{section}

\begin{document}
\title[sub-matrix summability of sequences of sets]{ sub-matrix summability of sequence of sets }
\author{\.{I}. Da\v{g}adur}
\address{Mersin University, Faculty of Science and Literature, Department of
Mathematics, 33343 Mersin - TURKEY.}
\email{ilhandagadur@yahoo.com}
\email{seydasezgek@gmail.com}
\author{\c{S}. SEZGEK}
\maketitle

\begin{abstract}
The main purpose of this paper is to introduce the concepts of Wijsman $C_{\lambda}$ statistical convergence, Wijsman $C_{\lambda}$ summability and Wijsman $\mathcal{I}$-$C_{\lambda}$ summability for sequence of sets by using Ces\`{a}ro submethod. Also, we establish some relations among the Wijsman $C_{\lambda}$ summability, Wijsman $D_{\lambda}$ summability and Wijsman $C_{1}$ summability. Finally, we have given some  equivalence results for these concepts.
\end{abstract}

~~  

\newline


\section{introduction and notations}

The concept of convergence of sequences of real numbers has been extended to statistical convergence independently by Fast \cite{fast} and Schoenberg \cite{schoenberg}. The idea of $\mathcal{I}$-convergence was introduced by Kostroyko et al. \cite{kostroyko} as a generalization of statistical convergence which is based on the structure of the ideal $\mathcal{I}$ of subset of the set of natural numbers. Recently, Das et al. \cite{das} introduced the notion of  $\mathcal{I}$-statistical convergence by using ideal. 

Connor \cite{connor} gave the relationships between the concepts of strongly $p$-Ces\`{a}ro summability and statistical convergence of sequences. 

The concept of convergence of sequences of numbers has been extended by several authors to convergence of sequences of sets. The one of these such extensions considered in this paper is the concept of Wijsman convergence (see \cite{baronti}, \cite{beer}, \cite{beer2}, \cite{nuray}, \cite{wijsman}, \cite{wijsman2}). Nuray and Rhoades \cite{nuray} studied statistical convergence set sequences and gave some basic theorems. Furthermore, the concept of strongly summable set sequences was given by \cite{nuray}. Wijsman $\mathcal{I}$-convergence by using ideal was introduced by Ki\c{s}i and Nuray \cite{kisi1}.

Now, we recall the basic definitions.

Let $x=\{x_k\}$ be a sequence of complex numbers. Then $x=\{x_k\}$ statistically convergent to $L$ provided that for every $\varepsilon>0$, 
\begin{equation*}
\lim_{n\to\infty}\frac{1}{n}|\{k\leq n : |x_k-L|\geq\varepsilon \}|=0 \ ,
\end{equation*}
where the vertical bars indicate the number of elements in the set \cite{fridy}.

A family of sets $\mathcal{I}\subseteq 2^{\mathbb{N}}$ is called an ideal if and only if 

\textit{i)} $\emptyset\in \mathcal{I}$, \textit{ii)} for each $A, B \in \mathcal{I}$ we have $A\cup B\in \mathcal{I}$, 
\textit{iii)} for each $A \in \mathcal{I}$ and each $B\subseteq A$ we have $B\in \mathcal{I}$.

An ideal is called non-trivial if $\mathbb{N}\in \mathcal{I}$ and non-trivial ideal is called admissible if $\{n\}\in\mathcal{I}$ for each $n\in\mathbb{N}$. 

A family of sets $\mathcal{F}\subseteq2^{\mathbb{N}}$ is a filter if and only if 
	
\textit{i)} $\emptyset\notin \mathcal{F}$, 
\textit{ii)} for each $A, B \in \mathcal{F}$ we have $A\cap B\in \mathcal{F}$, 
\textit{iii)} for each $A \in \mathcal{F}$ and each $A\subseteq B$ we have $B\in \mathcal{F}$.

In \cite{kostroyko}, $\mathcal{I}$ is non-trivial ideal in $\mathbb{N}$ if and only if 
\begin{equation*}
\mathcal{F}(\mathcal{I})=\{ M\subset \mathbb{N} : (\exists A\in\mathcal{I})(M=\mathbb{N}\backslash A)  \}
\end{equation*}
is a filter in $\mathbb{N}$.

Let  $\mathcal{I}\subset 2^{\mathbb{N}}$ be an admissible ideal. A sequence $x=\{x_k\}$ of elements of $\mathbb{R}$ is said to be $\mathcal{I}$-convergent to $L\in \mathbb{R}$ if for every $\varepsilon>0$ the set 
\begin{equation*}
A(\varepsilon)=\{ k\in\mathbb{N}: |x_k-L|\geq\varepsilon  \} \in \mathcal{I} \ .
\end{equation*}

Let $(X,\rho)$ be a metric space. For any point $x\in X$ and any non-empty subset $A$ of $X$, we define the distance from $x$ to $A$ by
\begin{equation*}
d(x,A)=\inf_{a\in A}\rho(x,a) \ .
\end{equation*}

Throughout this paper, we let $(X,\rho)$ be a metric space and $A$, $A_k$ be any non-empty closed subsets of $X$. 

The sequence $\{A_k\}$ is bounded if $\sup_{k}\{ d(x,A_k)  \}<\infty$ for each $x\in X$. The set of all bounded set sequences is denoted by $L_{\infty}$.

The sequence $\{A_k\}$ is Wijsman convergent to $A$ provided for each $x\in X$, we have
\begin{equation*}
\lim_{k\to\infty}d(x,A_k)=d(x,A)  \ . 
\end{equation*}

The sequence $\{A_k\}$ is Wijsman statistical convergent to $A$ if $\{ d(x,A_k) \}$ is statistical convergent to $d(x,A)$; i.e., for every $\varepsilon>0$ and for each $x\in X$, 
\begin{equation*}
\lim_{n\to\infty} \frac{1}{n} \left| \left\{ k\leq n : |d(x,A_k)-d(x,A)|\geq\varepsilon   \right\}  \right|=0     \ . 
\end{equation*}

The sequence $\{A_k\}$ is Wijsman Ces\`{a}ro summable to $A$ if for each $x\in X$,
\begin{equation*}
\lim_{k\to\infty} \frac{1}{n}\sum_{k=1}^{n}  d(x,A_k)=d(x,A)  \ . 
\end{equation*}

The sequence $\{A_k\}$ is Wijsman $\mathcal{I}$-convergent to $A$ if for every $\varepsilon>0$ and for each $x\in X$, 
\begin{equation*}
A(x,\varepsilon)= \left\{ k\in\mathbb{N} : |d(x,A_k)-d(x,A)|\geq\varepsilon   \right\}\in\mathcal{I}  \ .
\end{equation*}

The sequence $\{A_k\}$ is Wijsman $\mathcal{I}$-statistical convergent to $A$ if for every $\varepsilon>0$, $\delta>0$ and for each $x\in X$, 
\begin{equation*}
\left\{ n\in\mathbb{N} :  \frac{1}{n}  \left| \left\{  k\leq n :  |d(x,A_k)-d(x,A)|\geq\varepsilon  \right\}    \right|\geq\delta    \right\}\in\mathcal{I} \ .
\end{equation*}
In this case we write $ A_k \overset{ \mathcal{S}(\mathcal{I}_{\mathcal{W}}) }{\longrightarrow} A $.

In 1932, Agnew \cite{agnew} defined the deferred Ces\`{a}ro mean $D_{p,q}$ of the sequence $x$ by
\begin{equation*}
(D_{p,q}x)_n=\frac{1}{q(n)-p(n)}\sum_{k=p(n)+1}^{q(n)}x_k
\end{equation*} 
where $\{p(n)\}$ and $\{q(n)\}$ are sequences of nonnegative integers satisfying the conditions $p(n)<q(n)$ and $\lim_{n\to\infty}q(n)=\infty$. $D_{p,q}$ is clearly regular for any choise of $\{p(n)\}$ and $\{q(n)\}$.

Let $F$ be an infinite subset of $\mathbb{N}$ and $F$ as the range of a strictly increasing sequence of positive integers, say $F=\{ \lambda(n) \}$. The Ces\`{a}ro submethod $C_{\lambda}$ is defined as 
\begin{equation*}
(C_{\lambda }x)_{n}=\frac{1}{\lambda (n)}\sum_{k=1}^{%
	\lambda (n)}x_{k}  \  \   ~~ (n=1,2,\ldots).
\end{equation*}%
where $\{x_k\}$ is a sequence of a real or complex numbers. Therefore, the $C_{\lambda}$ method yields a subsequence of the  Ces\`{a}ro method $C_1$, and hence it is regular for any $\lambda$. $C_{\lambda}$ is obtained by deleting a set of rows from Ces\`{a}ro matrix. The basic properties of $C_{\lambda}$ method can be found in \cite{armitage},  \cite{osikiewicz} and \cite{goffman} . 

Let	$\lambda=\{ \lambda(n) \}$ be an increasing sequence of $\mathbb{N}$ and $x=\{x_k\}$ be a sequence. Osikiewicz \cite{osikiewicz} defined that $x$ is $C_{\lambda}$ statistical convergent to $L$ if, for $\forall \varepsilon >0$,
\begin{equation*}
\lim_{n\to\infty} \frac{1}{\lambda(n)}\left| \left\{ k\leq \lambda(n) : |x_k-L|\geq \varepsilon \right\} \right|=0   ~~ \  .
\end{equation*}

\section{Equivalence results for Wijsman $C_{\lambda}$ summability }

In the present section we shall give the definitions of Wijsman $C_{\lambda}$ summability and Wijsman $C_{\lambda}$ statistical convergence and examine some equivalence results. Also, we establish some relations between Wijsman $C_{\lambda}$ summability and Wijsman $C_{1}$ summability.

\begin{definition}
Let	$\lambda=\{ \lambda(n) \}$ be an increasing sequence of $\mathbb{N}$ and $\{A_k\}$ be a set sequence.	The sequence $\{A_k\}$ is Wijsman $C_{\lambda}$ summable to $A$ if for each $x\in X$, 
	\begin{equation*}
	\lim_{n\to\infty} \frac{1}{\lambda(n)}\sum_{k=1}^{\lambda(n)}d(x,A_k)=d(x,A)  \  .
	\end{equation*}
	In this case we write  $	A_k \overset{W_{C_{\lambda}} \ }{\ \longrightarrow} A $ .
\end{definition}

\begin{definition}
Let	$\lambda=\{ \lambda(n) \}$ be an increasing sequence of $\mathbb{N}$ with $\lambda(0)=0$ and $\{A_k\}$ be a set sequence.	The sequence $\{A_k\}$ is Wijsman $D_{\lambda}$ summable to $A$ if for each $x\in X$, 
	\begin{equation*}
	\lim_{n\to\infty} \frac{1}{\lambda(n)-\lambda(n-1)}\sum_{k=\lambda(n-1)+1}^{\lambda(n)}d(x,A_k)=d(x,A) \  .
	\end{equation*}
	This is denoted by $ 	A_k \overset{W_{D_{\lambda}} \ }{\ \longrightarrow} A $ .
\end{definition}

\begin{definition}
	Let	$\lambda=\{ \lambda(n) \}$ be an increasing sequence of $\mathbb{N}$ and $\{A_k\}$ be a set sequence. $\{A_k\}$ is Wijsman $C_{\lambda}$ statistical convergent to $A$ if, for $\forall \varepsilon >0$,
	\begin{equation*}
	\lim_{n\to\infty} \frac{1}{\lambda(n)}\left| \left\{ k\leq \lambda(n) : |d(x,A_k-d(x,A))|\geq \varepsilon \right\} \right|=0.
	\end{equation*}
	In this case we write  $ 	A_k \overset{st-(W_{C_{\lambda}}) \ }{\ \longrightarrow  \  } A $ .
\end{definition}

The theorem below gives us equivalence Wijsman ${C_{\lambda}}$ convergence with Wijsman Ces\`{a}ro convergence for bounded sequences.

\begin{theorem}
	\label{c1}
	Let $E=\{\lambda(n)\}$ be an infinite subset of $\mathbb{N}$ and $\{A_k\}$ be a bounded sequence. Then Wijsman ${C_{1}}$ convergence is equivalent to Wijsman ${C_{\lambda}}$ convergence if and only if $\limsup_{n\to \infty} \frac{\lambda(n+1)}{\lambda(n)}=1$.
\end{theorem}

\begin{proof}
	We shall apply the same technique found in \cite{osikiewicz}.
	Let $\limsup_{n\to \infty} \frac{\lambda(n+1)}{\lambda(n)}=1$ and $\{A_k\}$ be a bounded sequence. Then $\exists \alpha >0$ such that  $ d(x,A_k)<\alpha $ for all $k$.
		
	The sequence $\{A_k\}$ is Wijsman ${C_{\lambda}}$ summable to $A$.	
	Consider the set $F=\mathbb{N}\backslash E:= {\mu(n)}$. If $F$ is finite, then we can show that Wijsman ${C_{1}}$ convergence is equivalent to Wijsman ${C_{\lambda}}$ convergence; so assume $F$ is infinite. Then there exists an $N$ such that for $n\geq N$, $\mu (n)>\lambda(1)$. Since $E$ and $F$ are disjoint, for $n\geq N$ there exists an integer $m$ such that $\lambda(m)<\mu(n)<\lambda(m+1)$. We write $\mu(n)=\lambda(m)+j$, where $0<j<\lambda(m+1)-\lambda(m)$. Then, for $n>N$,
	\begin{eqnarray}
	&& \left| ({C_{\mu}}A)_n-({C_{\lambda}}A)_m  \right| =  \left| \frac{1}{\mu(n)}\sum_{k=1}^{\mu(n)}d(x,A_k)-\frac{1}{\lambda(m)}\sum_{k=1}^{\lambda(m)}d(x,A_k)\right|  \nonumber \\
	&&=\left| \frac{1}{\lambda(m)+j}\sum_{k=1}^{\lambda(m)+j}d(x,A_k)-\frac{1}{\lambda(m)}\sum_{k=1}^{\lambda(m)}d(x,A_k)\right| \nonumber \\
	&&=\left| \frac{1}{\lambda(m)+j}\sum_{k=1}^{\lambda(m)}d(x,A_k)+\frac{1}{\lambda(m)+j}\sum_{k=\lambda(m)+1}^{\lambda(m)+j}d(x,A_k)-\frac{1}{\lambda(m)}\sum_{k=1}^{\lambda(m)}d(x,A_k)\right| \nonumber \\
	&&=\left| \sum_{k=1}^{\lambda(m)} \left( \frac{1}{\lambda(m)+j}-\frac{1}{\lambda(m)}\right)d(x,A_k)+\frac{1}{\lambda(m)+j}\sum_{k=\lambda(m)+1}^{\lambda(m)+j}d(x,A_k)\right| \nonumber \\
	&&\leq \alpha \sum_{k=1}^{\lambda(m)}\frac{j}{\lambda(m)\lambda(m+j)}+ \alpha \frac{j}{\lambda(m+j)} \nonumber \\
	&&=\alpha\frac{j\lambda(m)}{\lambda(m)\lambda(m+j)}+\alpha\frac{j}{\lambda(m+j)} = 2\alpha \frac{j}{\lambda(m+j)}<2\alpha\frac{j}{\lambda(m)}  \ .\nonumber 
	\end{eqnarray}
	Since $0<j<\lambda(m+1)-\lambda(m)$,
	\begin{eqnarray}
	\left| \left( {C_{\mu}} A\right)_n - \left( {C_{\lambda}} A\right)_m\right|< 2\alpha\frac{j}{\lambda(m)}<2\alpha\frac{\lambda(m+1)-\lambda(m)}{\lambda(m)} =o(1) \ . \nonumber 
	\end{eqnarray}
	Thus, 
	\begin{eqnarray}
	0\leq \left| \left( {C_{\mu}} A\right)_n -d(x,A) \right|&\leq& \left| \left( {C_{\mu}} A\right)_n - \left( {C_{\lambda}} A\right)_m\right|+ \left| \left( {C_{\lambda}} A\right)_m -d(x,A) \right|  \nonumber  \\
	&=&o(1)+ o(1)=o(1)  \ .  \nonumber
	\end{eqnarray}
Therefore the sequence $({C_{1}} A)_n$ may be partitioned into two disjoint subsequences, namely $\left( {C_{\lambda}} A\right)_n=\left( {C_{1}} A\right)_{\lambda(n)}$ and $\left( {C_{\mu}} A\right)_n=\left( {C_{1}} A\right)_{\mu(n)}$, each having the common limit $A$. Thus, the sequence $\{A_k\}$ must be Wijsman $C_{1}$ summable to $A$ and hence Wijsman $C_{1}$ convergence and Wijsman $C_{\lambda}$ convergence equivalent for bounded sequences. 
\end{proof}	

Recall that if $\lim_{n\to \infty} \frac{n}{\lambda(n)}>0$, then $\lim_{n\to \infty}\frac{\lambda(n+1)}{\lambda(n)}=1$  \cite{osikiewicz}. Then the following corollary follows from Theorem \ref{c1} .

\begin{corollary}
	Let $E=\{\lambda(n)\}$ and $F=\{\mu(n)\}$ be infinite subsets of $\mathbb{N}$ with $\lim_{n\to \infty} \frac{n}{\lambda(n)}>0$ and $\lim_{n\to \infty} \frac{n}{\mu(n)}>0$. Then Wijsman $ {C_{\lambda}} $ convergence and Wijsman $ {C_{\mu}} $ convergence are equivalent for bounded sequences. 
\end{corollary}  

\begin{proof}
	$\lim_{n\to \infty}\frac{\lambda(n+1)}{\lambda(n)}=\lim_{n\to \infty}\frac{\mu(n+1)}{\mu(n)}=1$ by hypothesis. Hence by Theorem \ref{c1} Wijsman $ {C_{\lambda}} $ convergence, Wijsman $ {C_{\mu}} $ convergence and Wijsman $ {C_{1}} $ convergence are all equivalent for bounded sequences. 
\end{proof}

The next theorem presents a characterization about the Wijsman $C_{\lambda}$ convergence.

\begin{theorem}
	\label{te2}
	Let $E=\{\lambda(n)\}$ and $F=\{\mu(n)\}$ be infinite subsets of $\mathbb{N}$. If $\lim_{n\to \infty} \frac{\mu(n)}{\lambda(n)}=1$, then Wijsman $ {C_{\lambda}} $ convergence is equivalent to Wijsman $ {C_{\mu}} $ convergence for bounded sequences. 
\end{theorem}

\begin{proof}
	Again, we shall apply the same technique found in \cite{osikiewicz}.
	Let $\{A_k\}$ be a bounded sequence, then $\exists \alpha >0$ such that  $ d(x,A_k)<\alpha $ for all $k$. Consider the sequences $Q(n)=\max{\{\lambda(n), \mu(n)\}}$, $q(n)=\min{\{\lambda(n), \mu(n)\}}$. Since $\lim_{n\to \infty} \frac{\mu(n)}{\lambda(n)}=1$, $\lim_{n\to \infty} \frac{q(n)}{Q(n)}=1$. Then 
	\begin{eqnarray}
	&& \left| ({C_{\lambda}}A)_n-({C_{\mu}}A)_n\right|=  \left| \frac{1}{\mu(n)}\sum_{k=1}^{\mu(n)}d(x,A_k)-\frac{1}{\lambda(n)}\sum_{k=1}^{\lambda(n)}d(x,A_k)\right|  \nonumber \\
	&&=\left| \frac{1}{Q(n)}\sum_{k=1}^{Q(n)}d(x,A_k)-\frac{1}{q(n)}\sum_{k=1}^{q(n)}d(x,A_k)\right| \nonumber \\
	&&=\left| \frac{1}{Q(n)}\sum_{k=1}^{q(n)}d(x,A_k)+\frac{1}{Q(n)}\sum_{k=q(n)+1}^{Q(n)}d(x,A_k)-\frac{1}{q(n)}\sum_{k=1}^{q(n)}d(x,A_k)\right| \nonumber \\
	&&=\left| \sum_{k=1}^{q(n)} \left( \frac{1}{Q(n)}-\frac{1}{q(n)}\right)d(x,A_k)+\frac{1}{Q(n)}\sum_{k=q(n)+1}^{Q(n)}d(x,A_k)\right| \nonumber \\
	&&\leq \alpha \sum_{k=1}^{q(n)}  \frac{Q(n)-q(n)}{Q(n) q(n)} + \alpha \frac{Q(n)-q(n)}{Q(n)}  \nonumber \\
	&&= 2\alpha \frac{Q(n)-q(n)}{Q(n)}=2\alpha\left( 1-\frac{q(n)}{Q(n)}\right)=o(1) \ . \nonumber
	\end{eqnarray}	
	Hence if $\{A_k\}$ is Wijsman ${C_{\lambda}}$ summable to $A$,
	\begin{eqnarray}
	0\leq \left| \left( {C_{\mu}} A\right)_n -d(x,A) \right|&\leq& \left| \left( {C_{\mu}} A\right)_n - \left( {C_{\lambda}} A\right)_n\right|+ \left| \left( {C_{\lambda}} A\right)_n -d(x,A) \right|  \nonumber  \\
	&=&o(1)+ o(1)=o(1)  \ .  \nonumber
	\end{eqnarray}
	Similarly if $\{A_k\}$ is Wijsman ${C_{\mu}}$ summable to $A$,
	\begin{eqnarray}
	0\leq \left| \left( {C_{\lambda}} A\right)_n -d(x,A) \right|&\leq& \left| \left( {C_{\lambda}} A\right)_n - \left( {C_{\mu}} A\right)_n\right|+ \left| \left( {C_{\mu}} A\right)_n -d(x,A) \right|  \nonumber  \\
	&=&o(1)+ o(1)=o(1)  \ .  \nonumber
	\end{eqnarray}
	This proves the result.
\end{proof}

To see that $\lim_{n\to \infty} \frac{\mu(n)}{\lambda(n)}=1$ is not necessary condition in Theorem \ref{te2}, simply consider the sequences $\lambda(n)=n^2$, $\mu(n)=n^3$. Then $\lim_{n\to \infty} \frac{\lambda(n+1)}{\lambda(n)}=\lim_{n\to \infty} \frac{\mu(n+1)}{\mu(n)}=1$, and hence, by Theorem \ref{c1}, Wijsman $ {C_{\lambda}}$ convergence, Wijsman $ {C_{\mu}}$ convergence, and Wijsman $ {C_{1}}$ convergence are all equivalent for bounded sequences. However, $\lim_{n\to \infty} \frac{\mu(n)}{\lambda(n)}\neq1$.

In the Theorem \ref{c1}, with $\limsup_{n}\frac{\lambda(n+1)}{\lambda(n)}=1$ replaced by $\lim_{n}\frac{\lambda(n+1)}{\lambda(n)}=1$, the following result is easily obtained by Theorem \ref{te2}.

We now examine inclusion relationship between Wijsman $C_{\lambda}$ convergence and Wijsman $D_{\lambda}$ convergence.

\begin{theorem}
	Let $E=\{\lambda(n)\}$ be an infinite subset of $\mathbb{N}$ with $\lambda(0)=0$. Then 
	\begin{equation*}
  A_k \overset{W_{D_{\lambda}} \ }{\ \longrightarrow} A \  \Rightarrow \  	A_k\overset{W_{C_{\lambda}} \ }{\ \longrightarrow} A \ .
	\end{equation*} 	
\end{theorem}

\begin{proof}
	Let $A_k \overset{W_{D_{\lambda}} \ }{\ \longrightarrow} A$. Then, for any $n$, 
	\begin{eqnarray}
	&&( {C_{\lambda}}A)_n = \frac{1}{\lambda(n)}\sum_{k=1}^{\lambda(n)}|d(x,A_k)-d(x,A)| \nonumber \\
	&=& \frac{1}{\lambda(n)}\left(\sum_{k=1}^{\lambda(1)}|d(x,A_k)-d(x,A)|+\sum_{k=\lambda(1)}^{\lambda(2)}|d(x,A_k)-d(x,A)|\right. \nonumber \\
	&&\left.+\ldots + \sum_{k=\lambda(n-1)}^{\lambda(n)}|d(x,A_k)-d(x,A)|\right)  \nonumber \\
	&=& \frac{\lambda(1)-\lambda(0)}{\lambda(n)}\left(\frac{1}{\lambda(1)-\lambda(0)}\sum_{k=1}^{\lambda(1)}|d(x,A_k)-d(x,A)| \right)\nonumber \\
	&&+\frac{\lambda(2)-\lambda(1)}{\lambda(n)}\left(\frac{1}{\lambda(2)-\lambda(1)}\sum_{k=\lambda(1)}^{\lambda(2)}|d(x,A_k)-d(x,A)|\right) \nonumber \\
	&&+\ldots + \frac{\lambda(n)-\lambda(n-1)}{\lambda(n)}\left(\frac{1}{\lambda(n)-\lambda(n-1)}\sum_{k=\lambda(n-1)}^{\lambda(n)}|d(x,A_k)-d(x,A)|\right)  \nonumber \\
	&=& \frac{\lambda(1)-\lambda(0)}{\lambda(n)}({D_{\lambda}}A)_1+ \frac{\lambda(2)-\lambda(1)}{\lambda(n)}({D_{\lambda}}A)_2+\ldots +\frac{\lambda(n)-\lambda(n-1)}{\lambda(n)}({D_{\lambda}}A)_n \nonumber 
	\end{eqnarray}
	Let $T=(t_{nk})$ be the matrix defined by 
	\begin{equation*}
	t_{nk}=\left\{  \begin{array}{ccc}
	\frac{\lambda(k)-\lambda(k-1)}{\lambda(n)} &, & \text{for } k=1,2,\ldots  \\
	0 &,& \text{otherwise}  \  .
	\end{array}
	\right.
	\end{equation*}
	Clearly, $T$ is regular and we see that $( {C_{\lambda}}A)_n=(T({D_{\lambda}}A))_n$ . Since $\lim_{n\to\infty} ({D_{\lambda}}A)_n=d(x,A)$ and $T$ is regular, $\lim_{n\to\infty} (T({D_{\lambda}}A))_n=d(x,A)$. Hence, $\lim_{n\to\infty} ({C_{\lambda}}A)_n=d(x,A)$. This completes proof. 	
\end{proof}

\begin{theorem}
	Let $E=\{\lambda(n)\}$ be an infinite subset of $\mathbb{N}$ with $\lambda(0)=0$. Then
	\begin{equation*}
	A_k \overset{W_{C_{\lambda}} \ }{\ \longrightarrow} A \  \Rightarrow \  	A_k \overset{W_{D_{\lambda}} \ }{\ \longrightarrow} A \ 
	\end{equation*} 
	if and only if 	$\liminf_{n\to \infty} \frac{\lambda(n)}{\lambda(n-1)}>1$ .
\end{theorem}

\begin{proof}
	Let $	A_k \overset{W_{C_{\lambda}} \ }{\ \longrightarrow} A $. Then, for any $n$,
	\begin{eqnarray}
	&&({D_{\lambda}}A)_n =   \frac{1}{\lambda(n)-\lambda(n-1)}\sum_{k=\lambda(n-1)+1}^{\lambda(n)}d(x,A_k)           \nonumber \\
	&=& \frac{\lambda(n)}{\lambda(n)-\lambda(n-1)} \left(\frac{1}{\lambda(n)}\sum_{k=1}^{\lambda(n)}d(x,A_k) \right) \nonumber \\
	&&-\frac{\lambda(n-1)}{\lambda(n)-\lambda(n-1)} \left( \frac{1}{\lambda(n-1)}\sum_{k=1}^{\lambda(n-1)}d(x,A_k) \right) \nonumber  \\
	&=& \frac{\lambda(n)}{\lambda(n)-\lambda(n-1)} \left({C_{\lambda}} \right)_n 
	-\frac{\lambda(n-1)}{\lambda(n)-\lambda(n-1)} \left( {C_{\lambda}}  \right)_{n-1}  \ . \nonumber 
	\end{eqnarray}
	Let $R=(r_{nk})$ be the matrix defined by 
	\begin{equation*}
	r_{nk}=\left\{  \begin{array}{ccc}
	\frac{\lambda(n)}{\lambda(n)-\lambda(n-1)} &, & k=n  \\
	\frac{\lambda(n-1)}{\lambda(n)-\lambda(n-1)} &,& k=n-1  \\
	0 &,& \text{otherwise}  \  .
	\end{array}
	\right.
	\end{equation*}
	Thus, $( {D_{\lambda}}A)_n=(R({C_{\lambda}}A))_n$ and hence $
A_k \overset{W_{C_{\lambda}} \ }{\ \longrightarrow} A \  \Rightarrow \  	A_k \overset{W_{D_{\lambda}} \ }{\ \longrightarrow} A $
	if and only if 	$R$ is regular. $R$ will be regular if and only if the sequence 
	\begin{equation*}
	r_{nk}=\left\{ 
	\frac{\lambda(n)}{\lambda(n)-\lambda(n-1)}+\frac{\lambda(n-1)}{\lambda(n)-\lambda(n-1)} \right\}
	\end{equation*}
	is bounded. But, 
	\begin{eqnarray}
	\frac{\lambda(n)+\lambda(n-1)}{\lambda(n)-\lambda(n-1)}&=&1+\frac{2\lambda(n-1)}{\lambda(n)-\lambda(n-1)}  \nonumber \\
	&=& 1+\frac{2}{\frac{\lambda(n)}{\lambda(n-1)}-1} \nonumber 
	\end{eqnarray}
	and the last expression is bounded if and only if  $\liminf_{n\to \infty} \frac{\lambda(n)}{\lambda(n-1)}>1$ . Hence $A_k \overset{W_{C_{\lambda}} \ }{\ \longrightarrow} A \  \Rightarrow \  	A_k \overset{W_{D_{\lambda}} \ }{\ \longrightarrow} A$.
\end{proof}

Since Wijsman $C_{1}$ convergence implies Wijsman $C_{\lambda}$ convergence for any sequence $\{\lambda(n)\}$, we immediately have the following theorem.

\begin{theorem}
	Let $E=\{\lambda(n)\}$ be an infinite subset of $\mathbb{N}$. If $\{A_k\}$ is Wijsman statistical convergence to $A$, then $\{A_k\}$ is Wijsman $C_{\lambda}$ statistical convergence to $A$.
\end{theorem}

Theorem \ref{c1} immediately yields the following theorem
\begin{theorem}
	Let $E=\{\lambda(n)\}$ be an infinite subset of $\mathbb{N}$ and $\{A_k\}$ be a bounded sequence. If $\limsup_{n\to \infty} \frac{\lambda(n+1)}{\lambda(n)}=1$, then Wijsman statistical convergence is equivalent to Wijsman $C_{\lambda}$ statistical convergence. 
\end{theorem}

\section{ideal sub-matrix summability of sequences of sets}

In this section, the concepts of Wijsman  $\mathcal{I}$-$C_{\lambda}$ summability, Wijsman strongly $\mathcal{I}$-$C_{\lambda}$ summability and Wijsman strongly $\mathcal{I}$-$D_{\lambda}$ summability for sequence of sets are defined, and several theorems on this subjects are given.

\begin{definition}
Let	$\lambda=\{ \lambda(n) \}$ be an increasing sequence of $\mathbb{N}$ and $\{A_k\}$ be a set sequence.	The sequence $\{A_k\}$ is Wijsman  $\mathcal{I}$-$C_{\lambda}$ summable to $A$ if for every $\varepsilon >0$ and for each $x\in X$, 
	\begin{equation*}
	\left\{ n\in \mathbb{N} : \left|\frac{1}{\lambda(n)}\sum_{k=1}^{\lambda(n)}d(x,A_k)-d(x,A)\right|\geq \varepsilon  \right\}\in \mathcal{I} \  .
	\end{equation*}
	This is denoted by $	A_k \overset{C_{\lambda}(\mathcal{I}_W) }{ \longrightarrow} A$.
\end{definition}

\begin{definition}
Let	$\lambda=\{ \lambda(n) \}$ be an increasing sequence of $\mathbb{N}$ and $\{A_k\}$ be a set sequence.	The sequence $\{A_k\}$ is Wijsman strongly $\mathcal{I}$-$C_{\lambda}$ summable to $A$ if for every $\varepsilon >0$ and for each $x\in X$, 
	\begin{equation*}
	\left\{ n\in \mathbb{N} : \frac{1}{\lambda(n)}\sum_{k=1}^{\lambda(n)}\left|d(x,A_k)-d(x,A)\right|\geq \varepsilon  \right\}\in \mathcal{I} \  .
	\end{equation*}
In this case we write 	$	A_k \overset{C_{\lambda}[\mathcal{I}_W] }{ \longrightarrow} A$ . 
\end{definition}

\begin{definition}
Let	$\lambda=\{ \lambda(n) \}$ be an increasing sequence of $\mathbb{N}$ with $\lambda(0)=0$ and $\{A_k\}$ be a set sequence.	The sequence $\{A_k\}$ is Wijsman strongly $\mathcal{I}$-$D_{\lambda}$ summable to $A$ if for each $x\in X$, 
	\begin{equation*}
	\left\{ n\in \mathbb{N}: \frac{1}{\lambda(n)-\lambda(n-1)}\sum_{k=\lambda(n-1)+1}^{\lambda(n)}|d(x,A_k)-d(x,A)|\geq \varepsilon  \right\} \in \mathcal{I}  \  .
	\end{equation*}
In this case we write	$	A_k \overset{D_{\lambda}[\mathcal{I}_W] }{ \longrightarrow} A$ .
\end{definition}

\begin{theorem}
	Let $E=\{\lambda(n)\}$ be an infinite subset of $\mathbb{N}$ with $\lambda(0)=0$. If $\liminf_n\frac{\lambda(n)}{\lambda(n-1)}>1$, then 
	\begin{equation*}
	A_k \overset{C_{\lambda}[\mathcal{I}_W] }{ \longrightarrow} A \  \Rightarrow \  	A_k \overset{D_{\lambda}[\mathcal{I}_W] }{\ \longrightarrow} A \ .	
	\end{equation*}
\end{theorem}

\begin{proof}
	If $\liminf_n\frac{\lambda(n)}{\lambda(n-1)}>1$, then there exists $\beta>0$ such that $\frac{\lambda(n)}{\lambda(n-1)}\geq 1+\beta$ for all $n\in \mathbb{N}$. We have 
	\begin{equation*}
	\frac{\lambda(n)}{\lambda(n)-\lambda(n-1)}\leq \frac{1+\beta}{\beta} \ \text{  and  } \  \frac{\lambda(n-1)}{\lambda(n)-\lambda(n-1)}\leq \frac{1}{\beta} \ .
	\end{equation*}	
	Let $\varepsilon>0$ and we define the set 
	\begin{equation*}
	H=\left\{ n\in \mathbb{N}:\frac{1}{\lambda(n)}\sum_{k=1}^{\lambda(n)}\left|d(x,A_k)-d(x,A)\right|<\varepsilon\right\}
	\end{equation*}
	for each $x\in X$. We can say that $H\in \mathcal{F}(\mathcal{I})$. So we have 
	\begin{eqnarray}
	&&\frac{1}{\lambda(n)-\lambda(n-1)}\sum_{k=\lambda(n-1)+1}^{\lambda(n)}|d(x,A_k)-d(x,A)| \nonumber \\ &&=\frac{1}{\lambda(n)-\lambda(n-1)}\left(\sum_{k=1}^{\lambda(n)}|d(x,A_k)-d(x,A)|-\sum_{k=1}^{\lambda(n-1)}|d(x,A_k)-d(x,A)|\right) \nonumber \\
	&& =\frac{\lambda(n)}{\lambda(n)-\lambda(n-1)}\left(\frac{1}{\lambda(n)} \sum_{k=1}^{\lambda(n)}|d(x,A_k)-d(x,A)|\right) \nonumber \\
	&&- \frac{\lambda(n-1)}{\lambda(n)-\lambda(n-1)}\left(\frac{1}{\lambda(n-1)} \sum_{k=1}^{\lambda(n-1)}|d(x,A_k)-d(x,A)|\right) \nonumber \\
	&&< \frac{1+\beta}{\beta}\varepsilon_1 - \frac{1}{\beta}\varepsilon_2  \nonumber
	\end{eqnarray}
	for each $x\in X$ and $n\in H$. Choose $\varepsilon= \frac{1+\beta}{\beta}\varepsilon_1 - \frac{1}{\beta}\varepsilon_2$. Hence, for each $x\in X$
	\begin{equation*}
	\left\{ n\in \mathbb{N} : \frac{1}{\lambda(n)-\lambda(n-1)}\sum_{k=\lambda(n-1)+1}^{\lambda(n)}|d(x,A_k)-d(x,A)|<\varepsilon \right\} \in \mathcal{F}(\mathcal{I}) \  .
	\end{equation*}
	This completes the proof. 
\end{proof}

\begin{theorem}
	Let $E=\{\lambda(n)\}$ be an infinite subset of $\mathbb{N}$ with $\lambda(0)=0$. Then
	\begin{equation*}
	A_k \overset{D_{\lambda}[\mathcal{I}_W] }{ \longrightarrow} A  \  \Rightarrow \  	A_k \overset{C_{\lambda}[\mathcal{I}_W] }{ \longrightarrow} A  \ .
	\end{equation*}
\end{theorem}

\begin{proof}
	Let $		A_k \overset{D_{\lambda}[\mathcal{I}_W] }{ \longrightarrow} A  $ and we define the sets $G$ and $V$ such that
	\begin{equation*}
	G=\left\{n\in\mathbb{N}:\frac{1}{\lambda(n)-\lambda(n-1)}\sum_{k=\lambda(n-1)+1}^{\lambda(n)}|d(x,A_k)-d(x,A)|<\varepsilon_1 \right\}
	\end{equation*}
	and 
	\begin{equation*}
	V=\left\{n\in\mathbb{N}:\frac{1}{\lambda(n)}\sum_{k=1}^{\lambda(n)}|d(x,A_k)-d(x,A)|<\varepsilon_2 \right\} \ ,
	\end{equation*}
	for every $\varepsilon_1, \varepsilon_2>0$ and for each $x\in X$. Let 
	\begin{equation*}
	t_j=\frac{1}{\lambda(j)-\lambda(j-1)}\sum_{k=\lambda(j-1)+1}^{\lambda(j)}|d(x,A_k)-d(x,A)|<\varepsilon_1 
	\end{equation*}
	for $x\in X$ and for all $j\in G$. Clearly $G\in \mathcal{F}(\mathcal{I})$. 
	\begin{eqnarray}
	&&\frac{1}{\lambda(n)}\sum_{k=1}^{\lambda(n)}|d(x,A_k)-d(x,A)| \nonumber \\
	&=& \frac{1}{\lambda(n)}\left(\sum_{k=1}^{\lambda(1)}|d(x,A_k)-d(x,A)|+\sum_{k=\lambda(1)+1}^{\lambda(2)}|d(x,A_k)-d(x,A)|\right. \nonumber \\
	&&\left.+\ldots + \sum_{k=\lambda(n-1)+1}^{\lambda(n)}|d(x,A_k)-d(x,A)|\right)  \nonumber \\
	&=& \frac{\lambda(1)-\lambda(0)}{\lambda(n)}\left(\frac{1}{\lambda(1)-\lambda(0)}\sum_{k=1}^{\lambda(1)}|d(x,A_k)-d(x,A)| \right)\nonumber \\
	&&+\frac{\lambda(2)-\lambda(1)}{\lambda(n)}\left(\frac{1}{\lambda(2)-\lambda(1)}\sum_{k=\lambda(1)+1}^{\lambda(2)}|d(x,A_k)-d(x,A)|\right) \nonumber \\
	&&+\ldots + \frac{\lambda(n)-\lambda(n-1)}{\lambda(n)}\left(\frac{1}{\lambda(n)-\lambda(n-1)}\sum_{k=\lambda(n-1)+1}^{\lambda(n)}|d(x,A_k)-d(x,A)|\right)  \nonumber \\
	&=& \frac{\lambda(1)-\lambda(0)}{\lambda(n)}t_1+ \frac{\lambda(2)-\lambda(1)}{\lambda(n)}t_2+\ldots +\frac{\lambda(n)-\lambda(n-1)}{\lambda(n)}t_n \nonumber \\
	&\leq& \left(  \underset{j\in T}{\sup} \ t_j\right)<\varepsilon_1 \nonumber 
	\end{eqnarray}
	for each $x\in X$. By choosing $\varepsilon_2=\varepsilon_1$, we obtain $V\in \mathcal{F}(\mathcal{I})$. So, $	A_k \overset{C_{\lambda}[\mathcal{I}_W] }{ \longrightarrow} A  $. 
\end{proof}

\begin{theorem}
	Let $E=\{\lambda(n)\}$ be an infinite subset of $\mathbb{N}$ with $\lambda(0)=0$. If $\limsup_n\frac{\lambda(n)}{\lambda(n-1)}<\infty$, then 
	\begin{equation*}
	A_k \overset{D_{\lambda}[\mathcal{I}_W] }{ \longrightarrow} A  \  \Rightarrow \  	A_k \overset{C_{1}[\mathcal{I}_W] }{ \longrightarrow} A  \ .
	\end{equation*}
\end{theorem}

\begin{proof}
	Let $		A_k \overset{D_{\lambda}[\mathcal{I}_W] }{ \longrightarrow} A  $ and we define the sets $G$ and $V$ such that
	\begin{equation*}
	G=\left\{m\in\mathbb{N}:\frac{1}{\lambda(m)-\lambda(m-1)}\sum_{k=\lambda(m-1)+1}^{\lambda(m)}|d(x,A_k)-d(x,A)|<\varepsilon_1 \right\}
	\end{equation*}
	and 
	\begin{equation*}
	V=\left\{n\in\mathbb{N}:\frac{1}{n}\sum_{k=1}^{n}|d(x,A_k)-d(x,A)|<\varepsilon_2 \right\} \ ,
	\end{equation*}
	for every $\varepsilon_1, \varepsilon_2>0$ and for each $x\in X$. Let 
	\begin{equation*}
	t_j=\frac{1}{\lambda(j)-\lambda(j-1)}\sum_{k=\lambda(j-1)+1}^{\lambda(j)}|d(x,A_k)-d(x,A)|<\varepsilon_1 
	\end{equation*}
	for $x\in X$ and for all $j\in G$. Clearly $G\in \mathcal{F}(\mathcal{I})$. Choose an integer $\lambda(m-1)< n< \lambda(m)$ for $m\in G$. 
	\begin{eqnarray}
	&&\frac{1}{n}\sum_{k=1}^{n}|d(x,A_k)-d(x,A)| \nonumber \\
	&\leq& \frac{1}{\lambda(m-1)}\left(\sum_{k=1}^{\lambda(1)}|d(x,A_k)-d(x,A)|+\sum_{k=\lambda(1)+1}^{\lambda(2)}|d(x,A_k)-d(x,A)|\right. \nonumber \\
	&&\left.+\ldots + \sum_{k=\lambda(m-1)+1}^{\lambda(m)}|d(x,A_k)-d(x,A)|\right)  \nonumber \\
	&=& \frac{\lambda(1)-\lambda(0)}{\lambda(m-1)}\left(\frac{1}{\lambda(1)-\lambda(0)}\sum_{k=1}^{\lambda(1)}|d(x,A_k)-d(x,A)| \right)\nonumber \\
	&&+\frac{\lambda(2)-\lambda(1)}{\lambda(m-1)}\left(\frac{1}{\lambda(2)-\lambda(1)}\sum_{k=\lambda(1)+1}^{\lambda(2)}|d(x,A_k)-d(x,A)|\right) \nonumber \\
	&&+\ldots + \frac{\lambda(m)-\lambda(m-1)}{\lambda(m-1)}\left(\frac{1}{\lambda(m)-\lambda(m-1)}\sum_{k=\lambda(m-1)+1}^{\lambda(m)}|d(x,A_k)-d(x,A)|\right)  \nonumber \\
	&=& \frac{\lambda(1)-\lambda(0)}{\lambda(m-1)}t_1+ \frac{\lambda(2)-\lambda(1)}{\lambda(m-1)}t_2+\ldots +\frac{\lambda(m)-\lambda(m-1)}{\lambda(m-1)}t_m \nonumber \\
	&\leq& \left(  \underset{j\in T}{\sup} \ t_j\right)  \frac{\lambda(m)}{\lambda(m-1)}   \nonumber 
	\end{eqnarray}
	for each $x\in X$. Since $\limsup_n\frac{\lambda(n)}{\lambda(n-1)}<\infty$, we obtain $V\in \mathcal{F}(\mathcal{I})$. So, $	A_k \overset{C_{1}[\mathcal{I}_W] }{ \longrightarrow} A  $. 
\end{proof}

We now examine the relatonship between Wijsman $\mathcal{I}$-$C_{\lambda}$ statistical convergence and Wijsman $p$-strongly $\mathcal{I}$-$C_{\lambda}$ summability. 

\begin{definition}
Let	$\lambda=\{ \lambda(n) \}$ be an increasing sequence of $\mathbb{N}$ and $\{A_k\}$ be a set sequence.	The sequence $\{A_k\}$ is Wijsman $\mathcal{I}$-$C_{\lambda}$ statistical convergent to $A$ if for every $\varepsilon >0$, $\delta>0$ and for each $x\in X$, 
	\begin{equation*}
	\left\{  n\in \mathbb{N}: \frac{1}{\lambda(n)}\left| \left\{ k\leq \lambda(n) : |d(x,A_k)-d(x,A)|\geq \varepsilon  \right\}\right|\geq \delta   \right\}	\in \mathcal{I} \  .
	\end{equation*}
In this case we write $C_{\lambda}(\mathcal{I}_W)st-\lim A_k =A$ (or $ 	A_k \overset{st-C_{\lambda}(\mathcal{I}_W) }{ \longrightarrow} A  $) .
\end{definition}

\begin{definition}
Let	$\lambda=\{ \lambda(n) \}$ be an increasing sequence of $\mathbb{N}$ and $\{A_k\}$ be a set sequence.	The sequence $\{A_k\}$ is Wijsman $p$-strongly $\mathcal{I}$-$C_{\lambda}$ summable  to $A$ if for every $\varepsilon >0$,  for each $p$ positive real number and for each $x\in X$, 
	\begin{equation*}
	\left\{ n\in \mathbb{N} : \frac{1}{\lambda(n)}\sum_{k=1}^{\lambda(n)}\left|d(x,A_k)-d(x,A)\right|^p\geq \varepsilon  \right\}\in \mathcal{I} \  .
	\end{equation*}
\end{definition}

\begin{theorem}
	The sequence $\{A_k\}$ is  Wijsman $p$-strongly $\mathcal{I}$-$C_{\lambda}$ summable  to $A$ then it is Wijsman $\mathcal{I}$-$C_{\lambda}$ statistical convergent to $A$. 
\end{theorem}

\begin{proof}
	Let $\{A_k\}$ be Wijsman $p$-strongly $\mathcal{I}$-$C_{\lambda}$ summable  to $A$  and  given $\varepsilon>0$. Then we have 
	\begin{eqnarray}
	\sum_{k=1}^{\lambda(n)}|d(x,A_k)-d(x,A)|^p&&\geq \sum_{\substack{
			k=1,\\ |d(x,A_k)-d(x,A)|\geq \varepsilon} }^{\lambda(n)}|d(x,A_k)-d(x,A)|^p \nonumber \\
	&&\geq \varepsilon^p |\{k\leq \lambda(n) :  |d(x,A_k)-d(x,A)|\geq \varepsilon \}| \nonumber
	\end{eqnarray}	
	for each $x\in X$  and so 
	\begin{eqnarray}
	\frac{1}{\varepsilon^p  \lambda(n)}\sum_{k=1}^{\lambda(n)}|d(x,A_k)-d(x,A)|^p\geq \frac{1}{  \lambda(n)} |\{k\leq \lambda(n) :  |d(x,A_k)-d(x,A)|\geq \varepsilon \}| \ .\nonumber
	\end{eqnarray}	
	
	Hence, for given $\delta>0$ 
	\begin{eqnarray}
	&&\left\{n\in\mathbb{N} :  \frac{1}{  \lambda(n)} \left|\{k\leq \lambda(n) :  |d(x,A_k)-d(x,A)|\geq \varepsilon \}\right|\geq \delta  \right\}  \nonumber  \\
	&& \subseteq \left\{n\in\mathbb{N} :  \frac{1}{  \lambda(n)} \sum_{k=1}^{\lambda(n)}|d(x,A_k)-d(x,A)|^p\geq \varepsilon^p \delta  \right\} \in \mathcal{I} , \nonumber 
	\end{eqnarray} 
	for each $x\in X$. Therefore, $\{A_k\}$ is Wijsman $\mathcal{I}$-$C_{\lambda}$ statistical convergent to $A$. 
\end{proof}

\begin{theorem}
	Let the sequence $\{A_k\}$ be bounded. If $\{A_k\}$ is Wijsman $\mathcal{I}$-$C_{\lambda}$ statistical convergent to $A$ then it is Wijsman $p$-strongly $\mathcal{I}$-$C_{\lambda}$ summable  to $A$ 
\end{theorem}

\begin{proof}
	
	Suppose that $\{A_k\}$ is bounded. Then, there is an $\alpha>0$ such that $|d(x,A_k)-d(x,A)|<\alpha$ , for each $x\in X$ and for all $k$. Given $\varepsilon>0$, we have 
	
	\begin{eqnarray}
	\frac{1}{\lambda(n)}  \sum_{k=1}^{\lambda(n)}|d(x,A_k)-d(x,A)|^p&&= \frac{1}{\lambda(n)}\sum_{ \substack{
			k=1,\\ |d(x,A_k)-d(x,A)|\geq \varepsilon} }^{\lambda(n)}   |d(x,A_k)-d(x,A)|^p \nonumber \\
	&&+ \frac{1}{\lambda(n)}\sum_{ \substack{
			k=1,\\ |d(x,A_k)-d(x,A)|< \varepsilon} }^{\lambda(n)}|d(x,A_k)-d(x,A)|^p \nonumber \\
	&&\leq \frac{1}{\lambda(n)} \alpha^p  |\{k\leq \lambda(n):|d(x,A_k)-d(x,A)|\geq \varepsilon \}| \nonumber \\ 
	&&+ \frac{1}{\lambda(n)} \varepsilon^p |\{k\leq \lambda(n):|d(x,A_k)-d(x,A)|< \varepsilon \}| \nonumber \\
	&&\leq \frac{\alpha^p}{\lambda(n)}  |\{k\leq \lambda(n) :  |d(x,A_k)-d(x,A)|\geq \varepsilon \}| +\varepsilon^p . \nonumber
	\end{eqnarray}	
	
	Then, for any $\delta>0$ 
	\begin{eqnarray}
	&&\left\{n\in\mathbb{N} :  \frac{1}{  \lambda(n)} \sum_{k=1}^{\lambda(n)}|d(x,A_k)-d(x,A)|^p\geq  \delta  \right\}  \nonumber  \\
	&& \subseteq \left\{n\in\mathbb{N} :  \frac{1}{  \lambda(n)} \left|  \left\{ k\leq \lambda(n): |d(x,A_k)-d(x,A)|\geq \varepsilon \right\}  \right|\geq \frac{\delta^p}{\alpha^p}  \right\} \in \mathcal{I} , \nonumber 
	\end{eqnarray} 
	for each $x\in X$. Therefore $\{A_k\}$ is Wijsman $p$-strongly $\mathcal{I}$-$C_{\lambda}$ summable  to $A$ .	
\end{proof}


\begin{thebibliography}{00}
	
	\bibitem{armitage}  Armitage D.H. and Maddox I.J.,   \textit{A new type of Ces\`{a}ro mean},  Analysis, 9 (1989) 195-204.
	
	\bibitem{agnew}  Agnew R.P.,  \textit{On deferred Ces\`{a}ro means},  Annals of Math. 33[3] (1932) 413-421.
	
	\bibitem{osikiewicz}  Osikiewicz J.A.,  \textit{Equivalence results for Ces\`{a}ro submethods},  Analysis 20 (2000) 35-43. 
	
	\bibitem{dagadur}  Da\v{g}adur \.{I}., \textit{$C_{\lambda}$- conull FK spaces}, Demonstratio Math. 35[4] (2002) 835-848.
	
	\bibitem{aubin} Aubin J. P. and Frankowska H., \textit{Set valued analysis}, Birkhauser, Boston, 1990.
	
	\bibitem{baronti} Baronti M. and Papini P., \textit{convergence of sequences of sets}, In: Methods of functional analysis in approximation theory, ISNM 76 Birkhauser-Verlag, Basel, 1986 153-155. 
	
	\bibitem{beer} Beer G., \textit{On convergence of closed sets in a metric space and distance functions}, Bull. Aust. Math. Soc. 31, 1985, 421-432. 
	
	\bibitem{beer2} Beer G., \textit{Wijsman connvergence: A survey }, set valued var. anal. 2, 1994, 77-94.
	
	\bibitem{das} Das P., Sava\c{s} E. and Ghosal S. Kr., \textit{on generalized of certain summability methods using ideals}, Appl. Math. Letters, 36, 2011, 1509-1514.
	
	\bibitem{das2} Sava\c{s} E. and Das P., \textit{A generalized statistical convergence via idaels}, Appl. Math. Letters, 24, 2011, 826-830.
	
	\bibitem{freedman} Freedman A. R., Sember J. J. and Raphael M., \textit{Some Ces\`{a}ro-type summability spaces}, Proc. London Math. Soc., 37(3), 1978 508-520.
	
	\bibitem{fast} Fast H., \textit{Sur la convergence statistique }, Colloq. Math. 2, 1951, 241-244.
	
	\bibitem{goffman} Goffman C. and Petersen G. M., \textit{Submethods of regular matrix summability methods}, Canad. J. Math., 8, 1956, 40-46.
	
	\bibitem{connor} Connor J. S., \textit{The statistical and strong $p$-Ces\`{a}ro convergence of sequences}, Analysis, 8, 1988, 46-63.
	
	\bibitem{ulusu2} Ulusu U. and Ki\c{s}i \"{O}., \textit{ $\mathcal{I}$ - Ces\`{a}ro summability of sequences of sets }, Electronic Jour. of Mathematical Analysis and Applications, 5(1), 2017, 278-286.
	
	\bibitem{kisi1} Ki\c{s}i \"{O}. and Nuray F., \textit{New convergence definitions for sequences of sets}, Abstract and applied anaysis, 2013, Article ID 852796. 
	
	\bibitem{fridy} Fridy J. A., \textit{On statistical convergence}, Analysis, 5, 1985, 301-313. 
	
	\bibitem{kostroyko} Kostroyko P., \v{S}al\'{a}t T. and Wilczy\'{n}ski W., \textit{$\mathcal{I}$-convergence}, Real Anal. Exchange, 26(2), 2000, 669-686.
	
	\bibitem{nuray} Nuray F. and Rhoades B. E., \textit{statistical convergence of sequences of sets}, Fasc. Math. Letters, 49, 2012, 87-99.
	
	\bibitem{schoenberg} Schoenberg I. J., \textit{The integrability of certain functions and related summability methods}, Amer. Math. Monthly, 66, 1959, 361-375.
	
	
	
	\bibitem{wijsman} Wijsman R. A., \textit{Convergence of sequences of convex sets, cones and functions}, Bull. Amer. Math. Soc. 70,1964,186-188.
	
	
	\bibitem{wijsman2} Wijsman R. A., \textit{Convergence of sequences of convex sets, cones and functions II},Trans. Amer. Math. Soc. 123(1), 1966, 32-45.
	
\end{thebibliography}
\end{document}